\documentclass[11pt]{article}
\usepackage{amsmath,amssymb,amsthm,amscd,esint}

\numberwithin{equation}{section}

\DeclareMathOperator{\codim}{Codim}

\DeclareMathOperator{\vol}{vol}
\DeclareMathOperator{\dist}{dist}
\DeclareMathOperator{\tr}{tr}

\DeclareMathOperator{\diam}{diam}

\DeclareMathOperator{\Ric}{Ric}

\newtheorem{prop}{Proposition}[section]
\newtheorem{theo}[prop]{Theorem}
\newtheorem{lemm}[prop]{Lemma}
\newtheorem{clai}[prop]{Claim}
\newtheorem{coro}[prop]{Corollary}
\newtheorem{rema}[prop]{Remark}

\def\begeq{\begin{equation}}
\def\endeq{\end{equation}}

\begin{document}

\title{Convergence of K\"ahler-Ricci flow on lower dimensional\\ algebraic manifolds of general type}

\author{
Gang Tian\thanks{Supported partially by
NSF and NSFC grants. Email: tian@math.princeton.edu}\\
Beijing University and Princeton University
\\[5pt]
Zhenlei Zhang\thanks{Supported partially by NSFC 11431009 and Chinese Scholarship Council. Email: zhleigo@aliyun.com}\\
Capital Normal University}
\date{}

\maketitle

\begin{abstract}
In this paper, we prove that the $L^4$-norm of Ricci curvature is uniformly bounded along a K\"ahler-Ricci flow on any minimal algebraic manifold. As an application, we show that on any minimal algebraic manifold $M$ of general type and with dimension $n\le 3$,
any solution of the normalized K\"ahler-Ricci flow converges to the unique singular K\"ahler-Einstein metric on the canonical model of $M$ in the Cheeger-Gromov topology.
\end{abstract}

%\tableofcontents

%%%%%%%%%%%%%%%%%%%%%%%%%%%%%%%%%%%%%%%%%%%%%%%%%%%%%%%%%%%%%%%%%%%%%%%%%%%%%%%%%%%%%%%%%%%%%%%%%%%%%%%%%%%%%%%%%%%%%%%

%%%%%%%%%%%%%%%%%%%%%%%%%%%%%%%%%%%%%%%%%%%%%%%%%%%%%%%%%%%%%%%%%%%%%%%%%%%%%%%%%%%%%%%%%%%%%%%%%%%%%%%%%%%%%%%%%%%%%%%

\section{Introduction}

The purpose of this note is to prove the following

\begin{theo}\label{Main Thm}
Let $M$ be a smooth minimal model of general type with dimension $n\le 3$ and $\omega(t)$ be a solution to the normalized K\"ahler-Ricci flow
\begin{equation}
\frac{\partial}{\partial t}\omega=-\Ric-\omega.
\end{equation}
Then $(M,\omega(t))$ converges in the Cheeger-Gromov sense to the unique singular K\"ahler-Einstein metric on the canonical model of $M$.
\end{theo}

Here, by a smooth minimal model, we mean an algebraic manifold $M$ with nef canonical bundle $K_M$.
The theorem should remain true in higher dimensional case; cf. Conjecture 4.1 in \cite{So14-2}. Assuming the uniform bound of Ricci curvature, the conjecture is confirmed by Guo \cite{Gu15}. On the other hand, it has been known since Tsuji \cite{Ts88} and Tian-Zhang \cite{TiZh06} the convergence of the K\"ahler-Ricci flow in the current sense and the smooth convergence on the ample locus of the canonical class.

Applying the $L^2$ bound of Riemannian curvature (cf. Section 3) and K\"ahler-Einstein condition, we can say more about the limit singular space $M_\infty$.
When $n=2$, by a classical argument of removing isolated singularities following Anderson \cite{An89}, Bando-Kasue-Nakajima \cite{BKN} and Tian \cite{Ti90},
we can show that the limit is a smooth K\"ahler-Einstein orbifold with finite orbifold points. When $n=3$, by the argument of Cheeger-Colding-Tian \cite{ChCoTi} or Cheeger \cite{Ch03}, we have that the 2-dimensional Hausdorff measure of the singular set is finite. Moreover, it follows from the
parabolic version of the partial $C^0$-estimate (see \cite{TiZha13} that the limit $M_\infty$ is a normal variety and there is a natural holomorphic map from $M$ onto
$M_\infty$. This actually implies that $M_\infty$ is the canonical model of $M$.
%one can have any tangent cone of the limit space is a metric cone over a Sasaki-Einstein orbifold by Donaldson-Sun \cite{DoSu14}.

The proof of our theorem relies on a uniform $L^4$ bound of Ricci curvature under the K\"ahler-Ricci flow on a smooth minimal model together with the diameter boundedness of the limit singular K\"ahler-Einstein space (in the case of general type) proved by Song \cite{So14-1}. From these we derive a uniform local noncollapsing estimate of K\"ahler-Ricci flow on a minimal model of general type and the Gromov-Hausdorff convergence follows.

In Section 2, we present a short discussion on manifolds with integral bounded Ricci curvature, with emphasis on a uniform local noncollapsing estimate which is essential in extending the regularity theory of Cheeger-Conding and Cheeger-Colding-Tian. Then, in Section 3, we give a proof of our theorem by establishing a uniform $L^4$ Ricci curvature estimate under the K\"ahler-Ricci flow.

After we completed the first draft of this paper, we learned that 
Guo-Song-Weinkove obtained a different proof for the 2-dimensional case of Theorem \ref{Main Thm} (see \cite{guosongwein}).
\section{Manifolds with $L^p$ bounded Ricci curvature}

We recall the relative volume comparison formula relative to the $L^p$ Ricci curvature, $p$ bigger than half dimension, which is due to Petersen-Wei \cite{PeWe97}. It implies a uniform local noncollapsing property for the K\"ahler-Ricci flow on lower dimensional minimal models of general type. We will use this to prove the convergence of K\"ahler-Ricci flow on such manifolds.

Let $(M,g)$ be a complete Riemannian manifold of (real) dimension $m$. For any $\kappa\in\mathbb{R}$ denote by $B^\kappa_r$ a metric ball of radius $r$ in the space form of dimension $m$ with sectional curvature $\kappa$ and by $\vol(B^\kappa_r)$ its volume. Then we have

\begin{theo}[\cite{PeWe97}]
For any $p>\frac{m}{2}$ and $\Lambda<0$, there exists $C=C(m,p,\Lambda)$ such that the followings hold. First of all, for any $x\in M$ and $r>0$,
\begin{equation}
\frac{d}{dr}\bigg(\frac{\vol(B_r(x))}{\vol (B^\Lambda_r)}\bigg)^{\frac{1}{2p}}\le \frac{Cr^{2p}}{\vol(B^\Lambda_r)}\bigg(\int_{B_r(x)}\big|\big(\Ric-(m-1)\Lambda g\big)_-\big|^pdv\bigg)^{\frac{1}{2p}},
\end{equation}
where we define
$$\big(\Ric-(m-1)\Lambda g\big)_-=\max_{|v|=1}\big(0,-\Ric(v,v)+(m-1)\Lambda\big)$$
pointwisely. In particular, for any $r_2>r_1>0$,
\begin{equation}\label{volume comparison: 1}
\bigg(\frac{\vol(B_{r_2}(x))}{\vol (B^\Lambda_{r_2})}\bigg)^{\frac{1}{2p}}
-\bigg(\frac{\vol(B_{r_1}(x))}{\vol (B^\Lambda_{r_1})}\bigg)^{\frac{1}{2p}}
\le C\bigg(\int_{B_{r_2}(x)}\big|\big(\Ric-(m-1)\Lambda g\big)_-\big|^pdv\bigg)^{\frac{1}{2p}}.
\end{equation}
Then, by letting $r_1\rightarrow 0$ it gives, for any $r>0$,
\begin{equation}\label{volume comparison: 2}
\frac{\vol(B_r(x))}{\vol (B^\Lambda_r)}\le 1+C\int_{B_{r}(x)}\big|\big(\Ric-(m-1)\Lambda g\big)_-\big|^pdv.
\end{equation}
\end{theo}

The following corollary gives a kind of uniform local noncollapsing property on manifolds with integral Ricci curvature bound; see \cite{PeWe01} and \cite{DaWe} for similar volume doubling estimates.

\begin{coro}
For any $\Lambda<0$ and $p>\frac{m}{2}$, there exists $\varepsilon=\varepsilon(m,p,\Lambda)>0$ such that the following holds. Fix a base point $x_0\in M$. For any $x\in M$ with $\dist(x_0,x)=d$, if
\begin{equation}\label{Ricci condition: 1}
\frac{1}{\vol(B_1(x_0))}\int_{B_{2d+1}(x_0)}\big|\big(\Ric-(m-1)\Lambda g\big)_-\big|^pdv\le\frac{\varepsilon}{\vol(B^\Lambda_{d+1})},
\end{equation}
then,
\begin{equation}\label{volume comparison: 3}
\frac{\vol(B_r(x))}{r^m}\ge\frac{\vol(B_1(x_0))}{2\vol(B^\Lambda_{d+1})},\,\forall r\le 1.
\end{equation}
\end{coro}
\begin{proof}
By (\ref{volume comparison: 1}), for any $r\le 1$,
\begin{align}
\bigg(\frac{\vol(B_{r}(x))}{\vol (B^\Lambda_{r})}\bigg)^{\frac{1}{2p}}
&\ge\bigg(\frac{\vol(B_{d+1}(x))}{\vol (B^\Lambda_{d+1})}\bigg)^{\frac{1}{2p}}
 -C\bigg(\int_{B_{d+1}(x)}\big|\big(\Ric-(m-1)\Lambda g\big)_-\big|^pdv\bigg)^{\frac{1}{2p}}\nonumber\\
&\ge\bigg(\frac{\vol(B_{1}(x_0))}{\vol (B^\Lambda_{d+1})}\bigg)^{\frac{1}{2p}}
 -C\bigg(\int_{B_{2d+1}(x_0)}\big|\big(\Ric-(m-1)\Lambda g\big)_-\big|^pdv\bigg)^{\frac{1}{2p}}\nonumber.
\end{align}
where $C=C(m,p,\Lambda)$. If (\ref{Ricci condition: 1}) holds for some $\varepsilon=\varepsilon(m,p,\Lambda)$ sufficiently small, then
$$\bigg(\frac{\vol(B_{r}(x))}{\vol (B^\Lambda_{r})}\bigg)^{\frac{1}{2p}}\ge\frac{1}{2^{2p}}\bigg(\frac{\vol(B_{1}(x_0))}{\vol (B^\Lambda_{d+1})}\bigg)^{\frac{1}{2p}},$$
which is exactly the estimate (\ref{volume comparison: 3}).
\end{proof}

The lemma suggests a condition for Gromov-Hausdorff convergence. Suppose we have a sequence of complete Riemannian manifolds $(M_i,g_i)$ of dimension $m$ such that
\begin{equation}\label{Ricci condition: 2}
\int_M\big|\big(\Ric_{g_i}-(m-1)\Lambda g_i\big)_-\big|^pdv_{g_i}\rightarrow 0
\end{equation}
for some $p>\frac{m}{2}$, $\Lambda>0$, and
\begin{equation}\label{volume comparison: 4}
\vol_{g_i}(B_1(x_i))\ge v
\end{equation}
uniformly for some $v>0$ and $x_i\in M_i$, then along a subsequence, the manifolds $(M_i,g_i)$ are uniformly locally noncollapsing on $B_r(x_i)$ for any fixed $r>0$. Thus, we can apply Gromov precompactness theorem to get a noncollapsing limit of $(M_i,g_i,x_i)$ in the Gromov-Hausdorff topology. Furthermore, as showed in \cite{TiZha13}, one can extend the regularity theory of Colding \cite{Co97}, Cheeger-Colding \cite{ChCo96, ChCo97}, Cheeger-Colding-Tian \cite{ChCoTi} and Colding-Naber \cite{CoNa12} in this setting. If in addition we replace (\ref{Ricci condition: 2}) by
\begin{equation}\label{Ricci condition: 3}
\int_M\big|\Ric_{g_i}-(m-1)\Lambda g_i\big|^pdv_{g_i}\rightarrow 0,
\end{equation}
then Anderson's harmonic radius estimate \cite{An90} can also be applied. In summation, we can follow the arguments in \cite{PeWe01} and \cite{TiZha13} to prove

\begin{theo}\label{ChCo}
Let $(M_i,g_i)$ be a sequence of Riemannian manifolds satisfying {\rm(\ref{volume comparison: 4})} and {\rm(\ref{Ricci condition: 3})} for some $p>\frac{m}{2}$ and $\Lambda,v>0$. Then, passing to a subsequence if necessary, $(M_i,g_i,x_i)$ converges in the Cheeger-Gromov sense to a limit length space $(M_\infty,d_\infty,x_\infty)$ such that
\begin{itemize}
  \item[{\rm(1)}] for any $r>0$ and $y_i\in M_i$ with $y_i\rightarrow y_\infty\in M_\infty$ we have
  \begin{equation}
  \vol(B_r(y_i)\rightarrow\mathcal{H}^m(B_r(y_\infty)),
  \end{equation}
  where $\mathcal{H}^m$ denotes the $m$-dimensional Hausdorff measure;
  \item[{\rm(2)}] $M_\infty=\mathcal{R}\cup\mathcal{S}$ such that $\mathcal{S}$ is a closed set of codimension $\ge 2$ and $\mathcal{R}$ is convex in $M_\infty$; $\mathcal{R}$ consists of points whose tangent cones are $\mathbb{R}^m$;
  \item[{\rm(3)}] the convergence on $\mathcal{R}$ takes place in the $C^{\alpha}$ topology for any $0<\alpha<2-\frac{m}{p}$;
  \item[{\rm(4)}] the tangent of any $y\in M_\infty$ is a metric cone which splits off lines isometrically; the tangent cone has the same properties presented in {\rm(2)} and {\rm(3)};
  \item[{\rm(5)}] the singular set of $\mathcal{S}$ has codimension $\ge 4$ if each $M_i$ is K\"ahlerian.
\end{itemize}
\end{theo}

\section{K\"ahler-Ricci flow on minimal models}

\subsection{$L^4$ bound of Ricci curvature under K\"ahler-Ricci flow on minimal models}

Let $M$ be a smooth minimal model. If the Kodaira dimension equals 0, the manifold is Calabi-Yau and any K\"ahler-Ricci flow on $M$ converges smoothly to a Ricci flat metric \cite{Cao85}. We assume from now on that the Kodaira dimension of $M$ is positive. Then we consider the normalized K\"ahler-Ricci flow
\begin{equation}\label{KRF}
\frac{\partial}{\partial t}\omega(t)=-\Ric_{\omega(t)}-\omega(t),\,\omega(0)=\omega_0.
\end{equation}
It is well-known that the solution exists for all time $t\in[0,\infty)$ \cite{TiZh06}.

We shall prove the following theorem.

\begin{theo}
Suppose $M$ has positive Kodaira dimension and $K_M$ is semi-ample. Then there is a constant $C$ depending on $\omega_0$ such that
\begin{equation}\label{L4 estimate: 1}
\int_t^{t+1}\int_M|\Ric_{\omega(s)}|^4\omega(s)^nds\le C,\,\forall t\in[0,\infty).
\end{equation}
Moreover, for any $0<p<4$ we have
\begin{equation}\label{L4 estimate: 11}
\int_t^{t+1}\int_M|\Ric_{\omega(s)}+\omega(s)|^p\omega(s)^nds\rightarrow 0,\mbox{ as }t\rightarrow\infty.
\end{equation}
\end{theo}

We start with some general set-up following \cite{SoTi11}. Since $K_M$ is semi-ample, a basis of $H^0(M;K_M^\ell)$ for some large $\ell$ gives rise to a holomorphic map
\begin{equation}
\pi: M\rightarrow\mathbb{C}P^N,\,N=\dim H^0(M;K_M^\ell)-1.
\end{equation}
Let $\omega_{FS}$ be the Fubini-Study metric on $\mathbb{C}P^N$ and put
\begin{equation}
\chi=\frac{1}{\ell}\pi^*\omega_{FS}\in [K_M].
\end{equation}
Choose a smooth volume form $\Omega$ such that $\Ric(\Omega)=-\chi$.
Put
$$\hat{\omega}_t\,=\,e^{-t}\,\omega_0\,+\,(1-e^{-t})\chi.$$
It represents a K\"ahler metric in the class $[\omega(t)]$ and write
\begin{equation}
\omega(t)=\hat{\omega}_t+\sqrt{-1}\partial\bar{\partial}\varphi(t)
\end{equation}
for a family of smooth functions $\varphi(t)$. Then the K\"ahler-Ricci flow (\ref{KRF}) is equivalent to
\begin{equation}\label{MA}
\frac{\partial\varphi}{\partial t}=\log\frac{e^{(n-\kappa)t}(\hat{\omega}_t+\sqrt{-1}\partial\bar{\partial}\varphi)^n}{\Omega}-\varphi.
\end{equation}
%One immediate observation is
%\begin{equation}
%C_1^{-1}e^{-nt}\Omega\le\hat{\omega}_t^n\le C_1e^{-(n-\kappa)t}\Omega
%\end{equation}
%for some $C_1=C_1(\omega_0,\chi)$.

\begin{lemm}[\cite{SoTi11}]
There exists $C_i=C_i(\omega_0,\chi)$, $i=1,2$, such that
\begin{equation}\label{C0 estimate}
\|\varphi(t)\|_{C^0}+\|\varphi'(t)\|_{C^0}\le C_1,\,\forall t\ge 0
\end{equation}
and
\begin{equation}
\chi\le C_2\omega(t),\,\forall t\ge 0.
\end{equation}
\end{lemm}

Let $u=\varphi+\varphi'$ for any time $t$.

\begin{lemm}[\cite{SoTi11}]
There exists $C_3=C_3(\omega_0,\chi)$ such that
\begin{equation}
\|\nabla u(t)\|_{C^0}+\|\triangle u(t)\|_{C^0}\le C_3,\,\forall t\ge 0.
\end{equation}
\end{lemm}

When the manifold is of general type, these estimates are proved in \cite{Zh09}. We also observe that, from (\ref{MA}),
\begin{equation}
\Ric_{\omega(t)}+\sqrt{-1}\partial\bar{\partial}u(t)=-\chi.
\end{equation}
So, by the uniform bound of $\chi$ in terms of $\omega(t)$, to prove the $L^4$ bound of Ricci curvature, it suffices to show
\begin{equation}\label{L4 estimate: 2}
\int_t^{t+1}\int_M|\partial\bar{\partial}u(s)|^4\omega(s)^nds\le C_4,\,\forall t\ge 0
\end{equation}
for some $C_4$ independent of $t$. We follow the line in \cite{TiZha13} to prove this estimate.

As in \cite{TiZha13} we let $\nabla$ and $\bar{\nabla}$ denote the (1,0) and (0,1) part of the Levi-Civita connection of $\omega(t)$. Then, by the calculations in \cite{TiZha13}, we have the following lemmas

\begin{lemm}
There exists $C_5=C_5(\omega_0,\chi)$ such that
\begin{equation}
\int_M\big(|\nabla\nabla u|^2+|\nabla\bar{\nabla}u|^2+|\Ric|^2+|Rm|^2\big)\omega^n\le C_5,\,\forall t\ge 0.
\end{equation}
\end{lemm}

\begin{lemm}[\cite{TiZha13}]
There exists a universal constant $C_6=C_6(n)$ such that
\begin{equation}
\int_M|\nabla\bar{\nabla}u|^4\le C_6\int_M|\nabla u|^2\big(|\bar{\nabla}\nabla\nabla u|^2+|\nabla\nabla\bar{\nabla}u|^2\big);
\end{equation}
\begin{equation}
\int_M\big(|\bar{\nabla}\nabla\nabla u|^2+|\nabla\nabla\bar{\nabla}u|^2\big)\le C_6\int_M\big(|\nabla\triangle u|^2+|\nabla u|^2\cdot|Rm|^2\big).
\end{equation}
\end{lemm}

The estimates in the last lemma are general facts which remain hold for any smooth function.

The following theorem gives the required estimate (\ref{L4 estimate: 1}).

\begin{theo}
There exists $C_7=C_7(\omega_0,\chi)$ such that
\begin{equation}\label{L4 estimate: 3}
\int_t^{t+1}\int_M\big(|\nabla\bar{\nabla}u|^4+|\nabla\nabla\bar{\nabla}u|^2+|\bar{\nabla}\nabla\nabla u|^2\big)\le C_7,\,\forall t\ge 0.
\end{equation}
\end{theo}
\begin{proof}
By the previous two lemmas we are sufficient to prove a uniform $L^2$ bound of $\nabla\triangle u$. Recall the evolution of $\triangle u$, cf. (3.22) in \cite{SoTi11},
$$\frac{\partial}{\partial t}\triangle u=\triangle\triangle u+\triangle u+\langle\Ric,\partial\bar{\partial}u\rangle_\omega+\triangle\big(\tr_\omega\chi\big).$$
Thus,
\begin{equation}
\frac{\partial}{\partial t}(\triangle u)^2=\triangle(\triangle u)^2-2|\nabla\triangle u|^2+2(\triangle u)^2+2\triangle u\langle\Ric,\partial\bar{\partial}u\rangle+2\triangle u\triangle\big(\tr_\omega\chi\big).
\end{equation}
Integrating over the manifolds gives
\begin{align}
\int_M|\nabla\triangle u|^2
&\le\int_M\bigg((\triangle u)^2+|\triangle u||\Ric||\nabla\bar{\nabla}u|-\nabla_i\triangle u\nabla_{\bar{i}}\big(\tr_\omega\chi\big)-\frac{1}{2}\frac{\partial}{\partial t}(\triangle u)^2\bigg)\nonumber\\
&\le\int_M\bigg(\frac{1}{2}|\nabla\triangle u|^2+(\triangle u)^2+(\triangle u)^2|\Ric|^2+|\nabla\bar{\nabla}u|^2+2|\nabla\big(\tr_\omega\chi\big)|^2\bigg)\nonumber\\
&\hspace{1cm}-\frac{1}{2}\int_M(\triangle u)^2(s+n)-\frac{1}{2}\frac{d}{d t}\int_M(\triangle u)^2.\nonumber
\end{align}
Applying the uniform bound of $\triangle u$ and above lemma, we get
$$\int_M|\nabla\triangle u|^2\le C\int_M\big(1+|\nabla(\tr_\omega\chi)|^2\big)-\frac{d}{dt}\int_M(\triangle u)^2.$$
Integrating over the time interval we have
\begin{equation}\label{L4 estimate: 4}
\int_t^{t+1}\int_M|\nabla\triangle u|^2\le C\int_t^{t+1}\int_M\big(1+|\nabla(\tr_\omega\chi)|^2\big)+C,\,\forall t\ge 0.
\end{equation}
The term $|\nabla(\tr_\omega\chi)|^2$ can be estimated through the Schwarz lemma. Recall the following formula
\begin{equation}
\triangle\big(\tr_\omega\chi\big)\ge-|\Ric|\tr_\omega\chi-C(\tr_\omega\chi)^2
+\frac{|\nabla(\tr_\omega\chi)|^2}{\tr_\omega\chi}
\end{equation}
where $C$ is a universal constant given by the upper bound of the bisectional curvature of $\omega_{FS}$ on $\mathbb{C}P^n$. Because $0\le\tr_\omega\chi\le C$ under the flow, we have
$$|\nabla(\tr_\omega\chi)|^2\le C\bigg(\triangle\big(\tr_\omega\chi\big)+C|\Ric|+C\bigg).$$
Thus,
$$\int_M|\nabla(\tr_\omega\chi)|^2\le C\int_M(1+|\Ric|)\le C$$
uniformly. Substituting into (\ref{L4 estimate: 4}) we get the desired estimate.
\end{proof}

To prove (\ref{L4 estimate: 11}) we use the $L^2$ estimate to traceless Ricci curvature following the calculation by Y. Zhang \cite{ZhaY09}.

\begin{lemm}
Under the K\"ahler-Ricci flow,
\begin{equation}
\int_t^{t+1}\int_M|\Ric_{\omega(s)}+\omega(s)|^2\omega(s)^nds\rightarrow 0,\mbox{ as }t\rightarrow\infty.
\end{equation}
\end{lemm}
\begin{proof}
Recall the evolution of scalar curvature $R=\tr_\omega\Ric$, cf. \cite{ZhaY09},
$$\frac{\partial}{\partial t}R=\triangle R+|\Ric|^2+R=\triangle R+|\Ric+\omega|^2-(R+n).$$
The maximum principle shows that $\breve{R}=\inf R$ satisfies $\frac{d}{dt}\breve{R}\ge-(\breve{R}+n)$, which implies immediately
\begin{equation}\label{scalar curvature}
\breve{R}(t)+n\ge e^{-t}\min\big(\breve{R}(0)+n,0\big)\ge-Ce^{-t}
\end{equation}
for some $C=C(\omega_0)>0$. Then,
\begin{eqnarray}
\int|\Ric+\omega|^2\omega^n&=&\int\big(\frac{\partial}{\partial t}R+R+n\big)\omega^n\nonumber\\
&=&\frac{d}{dt}\int R\omega^n+\int(R+n)(R+1)\omega^n\nonumber\\
&=&\frac{d}{dt}\int R\omega^n+\int(R+n+Ce^{-t})(R+1)\omega^n-Ce^{-t}\int(R+1)\omega^n\nonumber\\
&\le&\frac{d}{dt}\int R\omega^n+C\int(R+n)\omega^n+Ce^{-t}\nonumber
\end{eqnarray}
where we used the uniform bound of scalar curvature and volume of $\omega(t)$. The integration of $R+n$ can be computed as
$$\int(R+n)\omega^n=n\int(\Ric+\omega)\wedge\omega^{n-1}=n\int(-\chi+\hat{\omega})\wedge\hat{\omega}^{n-1}
=ne^{-t}\int(\omega_0-\chi)\wedge\hat{\omega}^{n-1}.$$
Thus, $\int(R+n)\omega^n\le Ce^{-t}$. Then we have
$$\int_0^\infty\int|\Ric+\omega|^2\omega(t)^ndt\le\lim_{t\rightarrow\infty}\int R(t)\omega(t)^n-\int R(0)\omega_0^n+C\le C.$$
The lemma is proved by this estimate.
\end{proof}

The estimate (\ref{L4 estimate: 11}) when $2\le p<4$ then is a direct consequence of the H\"older inequality
$$\int_t^{t+1}\int_M|\Ric+\omega|^p
\le\bigg(\int_t^{t+1}\int_M|\Ric+\omega|^4\bigg)^{\frac{p-2}{2}}\bigg(\int_t^{t+1}
\int_M|\Ric+\omega|^2\bigg)^{\frac{4-p}{2}}.$$
When $0<p<2$ the estimate (\ref{L4 estimate: 11}) is obvious.

\begin{rema}
M. Simon presented in \cite{Si15} an $L^4$ Ricci curvature estimate under Ricci flow on a four-manifold. Combined with the K\"ahler condition, his arguments can be adapted to give an alternative proof
of our estimate. Another related integral bound of curvature can be found in \cite{BaZh15}.
\end{rema}

\subsection{Cheeger-Gromov convergence}

When the manifold $M$ is of general type, the K\"ahler-Ricci flow (\ref{KRF}) should converge in the Gromov-Hausdorff topology to a singular K\"ahler-Einstein metric on the canonical model; cf. Conjecture 4.1 in \cite{So14-2}. In this subsection we confirm this convergence without any curvature assumption in the case of the dimension less than or equal to 3.

Recall the holomorphic map $\pi:M\rightarrow\mathbb{C}P^N$ by a basis of $H^0(M;K_M^\ell)$ for some large $\ell$. Its image $M_{can}=\pi(M)$ is called the canonical model of $M$. Let $E\subset M$ be the exceptional locus of $\pi$. Then we have

\begin{theo}[\cite{TiZh06, So14-1}]\label{a prior theorem}
Let $M$ be a smooth minimal model of general type %of dimension equal to $2$ or $3$
and $\omega(t)$ be any solution to the K\"ahler-Ricci flow {\rm(\ref{KRF})}. Then,
\begin{itemize}
\item[{\rm(1)}] $\omega(t)$ converges in the current sense to a K\"ahler-Einstein metric $\omega_\infty$, the convergence takes place smoothly outside the exceptional locus $E$;

\item[{\rm(2)}] the metric completion of $(M\backslash E,\omega_\infty)$ is homeomorphic to $M_{can}$, so it is compact.
\end{itemize}
\end{theo}

\begin{rema}
It is known that the exceptional locus $E$ coincides with the non-ample or non-K\"ahler locus of the canonical class; cf. \cite{TiZh06} and \cite{CoTo13}.
\end{rema}

Suppose $\dim_{\mathbb{C}}M\le 3$. Let $t_i\rightarrow\infty$ be a sequence of times such that
\begin{equation}\label{L4 estimate: 6}
\int_M|\Ric_{\omega(t_i)}+\omega(t_i)|^{\frac{7}{2}}\omega(t_i)^n\rightarrow 0,\,\mbox{ as }i\rightarrow\infty.
\end{equation}
Choose a regular point $x_0\in M\backslash E$. The volume of the unit ball $\vol(B_{1,\omega(t_i)}(x_0))$ has a uniform lower bound. By the $L^p$ extension of Cheeger-Colding-Tian theory, Theorem 2.3, we may assume that $(M,\omega(t_i),x_0)$ converges in the Cheeger-Gromov sense to a limit metric space $(M_\infty,d_\infty,x_\infty)$. Since the metric $\omega(t)$ converges smoothly on $M\backslash E$, we may view $(M\backslash E,\omega_\infty)$ as a subset of $(M_\infty,d_\infty)$ through a locally isometric embedding.

Let $M_\infty=\mathcal{R}\cup\mathcal{S}$ the regular/singular decomposition of $M_\infty$.

\begin{lemm}
Suppose $x\in\mathcal{R}$, $0<\alpha<2-\frac{4n}{7}$. There exists $r=r(x)>0$ such that any $x_i\in M$ converging to $x$ has a holomorphic coordinate $(z^1,\cdots,z^n)$ on $B_{r,\omega(t_i)}(x_i)$ which satisfies
$$\frac{1}{2}\le(g_{k\bar{\ell}})\le 2,\,\|g_{k\bar{\ell}}\|_{C^\alpha}\le 2$$
where $g_{k\bar{\ell}}=\omega(t_i)(\frac{\partial}{\partial z^k},\frac{\partial}{\partial\bar{z}^\ell})$.
\end{lemm}
\begin{proof}
The existence of holomorphic coordinates is well-known. It can be constructed by a slight modification of the local harmonic coordinates. We include a short proof for the reader's convenience. First of all, by the $C^\alpha$ convergence on $\mathcal{R}$, there is a sequence of harmonic coordinate $v_i=(v_i^1,\cdots,v_i^{2n})$ on $B_{r,\omega(t_i)}(x_i)$ for some $r>0$ independent of $i$ such that
$$\frac{3}{4}\le(h_{pq})\le\frac{4}{3}\mbox{ and }r^{-\alpha}\|h_{pq}\|_{C^\alpha}\le\frac{4}{3},$$
where $h_{pq}=g_i(\frac{\partial}{\partial v_i^p},\frac{\partial}{\partial v_i^q})$, $g_i$ is the K\"ahler metric of $\omega(t_i)$, and \cite{PeWe01}
$$\int_{B_{r,\omega(t_i)}(x_i)}|\langle\frac{\partial}{\partial v^p},\frac{\partial}{\partial v^q}\rangle-\delta_{pq}|\omega(t_i)^n\rightarrow 0,\,\mbox{ as }i\rightarrow \infty.$$
Moreover, since each $\omega(t_i)$ is K\"ahler, we may assume that
$$\int_{B_{r,\omega(t_i)}(x_i)}|\langle J\frac{\partial}{\partial v^p},\frac{\partial}{\partial v^{n+q}}\rangle-\delta_{pq}|\omega(t_i)^n\rightarrow 0,\,\mbox{ as }i\rightarrow \infty,$$
for any $1\le p,q\le n$, where $J$ is the complex structure of $M$; see \cite[Section 9]{ChCoTi} for a discussion. Then we choose a pseudoconvex domain $B_{\frac{1}{2}r,\omega(t_i)}(x_i)\subset\Omega_i\subset B_{r,\omega(t_i)}(x_i)$ and solve the $\bar{\partial}$ equation in $\Omega_i$:
$$\bar{\partial}f_i^p=\bar{\partial}(v_i^p+\sqrt{-1}v^{n+p}),\,p=1,\cdots,n.$$
The domain can be chosen as the Euclidean ball in the local coordinate. The equation has solution satisfying the $L^2$ estimate \cite[Theorem 4.3.4]{ChSh}
$$\int_{B_{\frac{1}{2}r,\omega(t_i)}(x_i)}|f_i^p|^2\omega(t_i)^n\le Cr^2\int_{B_{r,\omega(t_i)}(x_i)}\big|\bar{\partial}(v_i^p+\sqrt{-1}v^{n+p})\big|^2\omega(t_i)^n$$
for a universal constant $C$. This implies $\int_{B_{\frac{1}{2}r,\omega(t_i)}(x_i)}|f_i^p|^2\omega(t_i)^n\rightarrow 0$ for all $1\le p\le n$. Then applying the elliptic regularity to $\triangle_{\omega(t_i)}f_i^p=0$ we get
$$\sup_{B_{\frac{1}{4}r,\omega(t_i)}(x_i)}\big(|\partial f_i^p|+|\bar{\partial}f_i^p|\big)\rightarrow 0,\,\forall 1\le p\le n.$$
In particular, the function $w_i^p=v^p+\sqrt{-1}v^{n+p}-f_i^p$, $1\le p\le n$, gives rise to the desired holomorphic coordinate on $B_{\frac{1}{4}r,\omega(t_i)}(x_i)$ whenever $i$ is large enough.
\end{proof}

\begin{lemm}
$M_\infty\backslash(M\backslash E)$ is a closed subset of $M_\infty$ with Hausdoeff codimension $\ge 2$.
\end{lemm}
\begin{proof}
Notice that $M\backslash E\subset\mathcal{R}$, $M_\infty\backslash(M\backslash E)=\mathcal{S}\cup\big(\mathcal{R}\backslash(M\backslash E)\big)$ where $\codim\mathcal{S}\ge 4$ by Theorem 2.3 (5). Therefore, it suffices to show that $\codim\big(\mathcal{R}\backslash(M\backslash E)\big)\ge 2$.

For any $x\in\mathcal{R}\backslash(M\backslash E)$ there exists a sequence of points $x_i\in E$ which converges to $x$. By above lemma, there exists local holomorphic coordinate in $B_{r,\omega(t_i)}(x_i)$ for some $r=r(x)>0$ independent of $i$ with required $C^\alpha$ estimate. The intersection $E\cap B_{r,\omega(t_i)}(x_i)$ is a local subvariety with finite volume, so passing to a subsequence, $E\cap B_{r,\omega(t_i)}(x_i)$ converges to a limit analytical set $E_\infty\subset B_{r,\omega_\infty}(x)$. Thus, $\mathcal{R}\backslash(M\backslash E)$ is an analytical set, $\codim \big(\mathcal{R}\backslash(M\backslash E)\big)\ge 2$ as desired.
\end{proof}

\begin{lemm}
$(M_\infty,d_\infty)$ is isometric to the metric completion of $(M\backslash E,\omega_\infty)$.
\end{lemm}
\begin{proof}
The lemma follows from the fact that $\codim\big(\mathcal{R}\backslash(M\backslash E)\big)\ge 2$ and the local isoperimetric constant estimate; see \cite{ChCo00} or \cite{RoZh11} for details. For the estimate of local isoperimetric constant, one can apply the same argument as Croke \cite{Cr80} by using the volume comparison of geodesic balls by Petersen-Wei (cf. \cite[Corolarry 2.4]{TiZha13} or (\ref{volume comparison: 2})).
\end{proof}

By the compactness of the limit space by Song \cite{So14-1}, see Theorem \ref{a prior theorem} (2) above, the diameters of the sequence $(M,\omega(t_i))$ are uniformly bounded.

\begin{lemm}
The K\"ahler-Ricci flow $\omega(t)$ is uniformly noncollapsing in the sense that: there exists $\kappa=\kappa(n,\omega_0)>0$ such that
\begin{equation}\label{noncollapsing}
\vol_{\omega(t)}\big(B_{r,\omega(t)}(x)\big)\ge\kappa r^{2n},\,\forall x\in M,r\le 1.
\end{equation}
\end{lemm}
\begin{proof}
The lemma follows from Perelman's noncollapsing estimate to Ricci flow \cite{Pe02} together with the scalar curvature estimate by Z. Zhang \cite{Zh09}. Suppose that
\begin{equation}\label{collapsing}
r_i^{-2n}\vol_{\omega(t_i)}\big(B_{r_i,\omega(t_i)}(x_i)\big)\rightarrow 0
\end{equation}
for a sequence of times $t_i\rightarrow\infty$ and $x_i\in M$, $r_i\le 1$. Choose $t_i'\in[t_i-2,t_i-1]$ such that (\ref{L4 estimate: 6}) hold at $t_i'$. Then by above lemma we have $(M,\omega(t_i'))$ converges in the Gromov-Hausdorff topology to the unique limit $(M_\infty,d_\infty)$. In particular we have
$$R(\omega(t_i'))\le C,\,\diam(M,\omega(t_i'))\le C,\,C_S(M,\omega(t_i'))\le C$$
for some $C$ independent of $i$, where $R$ is the scalar curvature, $C_S$ is the Sobolev constant.

Let $\tilde{\omega}_i(\tilde{t})=(1+\tilde{t})\omega\big(t_i'+\log(1+\tilde{t})\big)$ be the solution to the Ricci flow
\begin{equation}
\frac{\partial}{\partial\tilde{t}}\tilde{\omega}_i=-\Ric(\tilde{\omega}_i),\,\tilde{\omega}_i(0)=\omega(t_i').
\end{equation}
Under this rescalings, the metric $\omega(t_i)=(1+\tilde{t}_i)^{-1}\tilde{\omega}_i(\tilde{t}_i)$ for some $\tilde{t}_i\in[e-1,e^2-1]$ and $B_{r_i,\omega(t_i)}(x_i)=B_{\tilde{r}_i,\tilde{\omega}(\tilde{t}_i)}(x_i)$ for some $\tilde{r}_i\le e$. Recall the $\mu$ functional of Perelman \cite{Pe02}
$$\mu(g,\tau)=\inf\int_M\big[\tau(R+|\nabla f|^2)+f-2n\big](4\pi\tau)^{-n}e^{-f}dv_g$$
for any Riemannian metric $g$ and $\tau>0$, where the infimum is taken over all $f\in C^\infty(M;\mathbb{R})$ with restriction $\int_M(4\pi\tau)^{-n}e^{-f}dv_g=1$. The condition (\ref{collapsing}) implies that $\mu(\tilde{\omega}_i(\tilde{t}_i),\tilde{r}_i^2)\rightarrow -\infty$ as $i\rightarrow \infty$; see \cite[Section 4.1]{Pe02}. Then Perelman's monotonicity formula shows $\mu(\tilde{\omega}_i(0),\tilde{t}_i+\tilde{r}_i^2)\rightarrow -\infty$ as $i\rightarrow \infty$, where $\tilde{t}_i+\tilde{r}_i^2\in[e-1,2e^2-1]$. But this can never happen because of the lower estimate of $\mu$ (cf. the estimate in \cite{Zha07}):
\begin{equation}
\mu(\tilde{\omega}_i(0),\tau)\ge\tau\inf R(\tilde{\omega}_i(0))-\frac{n}{2}\ln\tau-n\ln C_S(\tilde{\omega}_i(0))-C(n),\,\forall\tau\ge\frac{n}{8}.
\end{equation}
So (\ref{collapsing}) can not hold, the lemma is proved.
\end{proof}

%Applying Perelman's monotonicity of the $\mu$ functional and some classical analysis as in \cite{Zh07} one can also show

%\begin{lemm}
%The Sobolev constant of $(M,\omega(t))$ admits a uniform upper bound, i.e.,
%\begin{equation}
%\bigg(\int_M f^{\frac{2n}{n-1}}\omega(t)^n\bigg)^{\frac{n-1}{n}}\le C\int_M\big(|\nabla f|^2+f^2\big)\omega(t)^n,\,\forall t\ge 0,f\in C^\infty(M;\mathbb{R}),
%\end{equation}
%where $C$ is independent of $t$.
%\end{lemm}

The global Cheeger-Gromov convergence is a direct corollary of the following proposition.

\begin{prop}
\label{pro:zzl-5}
For any sequence $t_j\rightarrow\infty$, $(M,\omega_{t_j})$ converges along a subsequence in the Cheeger-Gromov sense to the limit $(M_\infty,d_\infty)$.
\end{prop}
\begin{proof}
By the regularity of manifolds with $L^p$ bounded Ricci curvature, Theorem \ref{ChCo}, and the uniform $L^p$ estimate of Ricci curvature under the K\"ahler-Ricci flow, (\ref{L4 estimate: 1}) and (\ref{L4 estimate: 11}), we can find another sequence of times $t_j'$ such that
$$t_j-\varepsilon_j\le t_j'\le t_j$$
where $\varepsilon_j\rightarrow 0$ as $j\rightarrow \infty$, and
$$(M,\omega_{t_j'})\stackrel{d_{GH}}{\longrightarrow}(M_\infty,d_\infty)$$
along a subsequence. On the other hand, by Gromov precompactness theorem \cite{Gr} together with the local noncollapsing estimate (\ref{noncollapsing}), after passing to a subsequence, the manifolds $(M,\omega(t_j))$ also converge in the Gromov-Hausdorff topology to a compact limit
$$(M,\omega(t_j))\stackrel{d_{GH}}{\longrightarrow}(M_\infty',d_\infty').$$
It is remained to show that $(M_\infty',d_\infty')$ is isometric to $(M_\infty,d_\infty)$. Actually we have

\begin{clai}
There is a sequence of positive numbers $\delta_j\rightarrow 0$ as $j\rightarrow\infty$ such that the identity map defines a $\delta_j$-Gromov-Hausdorff approximation from $(M,\omega(t_j'))$ to $(M,\omega(t_j))$.
\end{clai}
\begin{proof}[Proof of the Claim]
Recall that by the smooth convergence of $\omega(t)$ on $M\backslash E$ and uniform volume noncollapsing (\ref{noncollapsing}) we have for any $\delta>0$ one compact subset $K\subset M\backslash E$ and $\epsilon>0$, $j_0>>1$ such that
\begin{equation}\label{e11}
\epsilon\le\inf_t\dist_{\omega(t)}(K,E)\le\sup_t\sup_{x\in E}\dist_{\omega(t)}(x,K)\le\delta,
\end{equation}
and,
\begin{equation}\label{e12}
\vol_{\omega(t)}(M\backslash K)\le\delta,\,\forall t\ge t_{j_0}.
\end{equation}
The later can be seen by simply the derivative estimate to volume
$$\frac{d}{dt}\int_{M\backslash K}\omega(t)^n=-\int_{M\backslash K}(\inf R+n)\omega(t)^n\le C(\omega_0)e^{-t}\int_{M\backslash K}\omega(t)^n$$
where we used the estimate for the scalar curvature (\ref{scalar curvature}). For example, one may choose $j_0$ and $K$ such that $C(\omega_0)e^{-j_0}\le\ln 2$ and $\vol_{\omega(t_{j_0})}(M\backslash K)\le\frac{\delta}{2}$, then the volume estimate holds for any $t\ge t_{j_0}$. Moreover, by the convexity of the regular set $M\backslash E$ in $M_\infty$ we may assume in addition that any minimal geodesic in $(M,\omega(t))$ with endpoints in $K$ does not intersect the $\epsilon$ neighborhood of $E$. Also observe that there exists $C=C(\epsilon)$ independent of $j$ such that
\begin{equation}\label{e13}
|\Ric(\omega(t))|(x)\le C,\,\forall \dist_{\omega(t)}(x,E)\ge\frac{\epsilon}{2}.
\end{equation}
Then, by the derivative estimate to distance function, cf. Lemma 8.3 in \cite{Pe02},
$$\frac{d}{dt}\dist_{\omega(t)}(x,y)\ge -2(2n-1)(C\epsilon+\epsilon^{-1}),\,\forall x,y\in K.$$
Thus, we obtain the distance lower variation estimate
$$\dist_{\omega(t_j)}(x,y)\ge\dist_{\omega(t_j')}(x,y)-2(2n-1)(C\epsilon+\epsilon^{-1})\varepsilon_j
\ge\dist_{\omega(t_j')}(x,y)-\sqrt{\varepsilon_j}$$
whenever $j$ is large enough. On the other hand, let $\gamma$ be any minimal geodesic in $(M,\omega(t_j'))$ connecting $x,y\in K$. By assumption $\dist_{\omega(t)}(\gamma,E)\ge\epsilon$ for any $t$, so
$$\frac{d}{dt}\int_\gamma|\dot{\gamma}|ds\le\int_\gamma|\Ric+\omega||\dot{\gamma}|ds\le (C+n)\int_\gamma|\dot{\gamma}|ds$$
which gives the upper estimate to the distance distortion
$$\dist_{\omega(t_j)}(x,y)\le e^{(C+n)\varepsilon_j}\dist_{\omega(t_j')}(x,y)\le\dist_{\omega(t_j')}(x,y)+\sqrt{\varepsilon_j}$$
whenever $j$ is large enough. Finally, (\ref{e11}) shows that $K$ is $\delta$-dense in any $(M,\omega(t))$. Combining with the upper and lower distance variation estimate we get that the identity map defines an $\delta$-Gromov-Hausdorff approximation between $(M,\omega(t_j))$ and $(M,\omega(t_j'))$ whenever $j$ is large enough such that $\sqrt{\varepsilon_j}\le\delta$. The claim is proved.
\end{proof}

Proposition \ref{pro:zzl-5} follows directly from the Claim.
\end{proof}

%\begin{rema}
%We can also apply the almost Einstein estimate {\rm(\ref{L4 estimate: 11})} to prove that $(M_\infty',d_\infty')$ is isometric to $(M_\infty,d_\infty)$.
%\end{rema}

We end the paper with a brief discussion about the algebraic structure of $M_\infty$. By Song's work \cite{So14-1}, the limit space $M_\infty$ coincides with the canonical model $M_{can}$, so it is a normal projective variety. On the other hand, using the K\"ahler-Ricci flow, we can 
produce a natural isomorphism from $M_\infty$ to $M_{can}$ through holomorphic sections of $K_M^\ell$ for some $\ell>>1$ such that $K_M^\ell$ is base point free, and consequently, give an alternative proof of the above result by J. Song.
This can be done by choosing a family of orthonormal basis of $H^0(M,K_M^\ell)$ with respect to the Hermitian metric $h(t)=\omega(t)^{-n\ell}$, say $\{s_{i,t}\}_{i=0}^{N_\ell}$
where $N_\ell=\dim H^0(M;K_M^\ell)-1$, which satisfies the ODE system
\begin{equation}
\frac{\partial}{\partial t}s_{i,t}=\sum_j b_{ij}(t)s_{j,t}.
\end{equation}
In order to preserve the orthonormal property we choose
\begin{equation}
b_{ij}=\bar{b}_{ji}=-\frac{\ell-1}{2}\int_M(R+n)\langle s_{i,t},s_{j,t}\rangle_{h(t)}\omega(t)^n.
\end{equation}

\begin{lemm}
There exists $C=C(\omega_0,\ell)$ such that
$$|b_{ij}|\le Ce^{-t}.$$
\end{lemm}
\begin{proof}
First of all we notice that the section $s_{i,t}$ admits a uniform $L^\infty$ bound. This can be seen by the uniform equivalence of the Hermitian metric $h(t)=e^{-\ell u(t)}\Omega^{-\ell}$, the volume form $\omega(t)^n=e^{u(t)}\Omega$ and the $L^\infty$ estimate of holomorphic sections at time $t=0$. Here we used the uniform $C^0$ estimate of $u$ (\ref{C0 estimate}). Then we can estimate the integration as in the proof of Lemma 3.7, by $R+n\ge-Ce^{-t}$ and $\int(R+n)\omega^n\le Ce^{-t}$,
$$|b_{ij}|\le C\int|R+n|\omega^n\le C\int(R+n+Ce^{-t})\omega^n+Ce^{-t}\le Ce^{-t}$$
where $C=C(\omega_0,\ell)$.
\end{proof}

Suppose $s_{i,t}=\sum_ja_{ij}(t)s_{j,0}$ and denote by $A(t)=(a_{ij}(t))$ a Hermitian matrix, then $A(t)=e^{\int_0^tB(s)ds}$ where $B(t)=(b_{ij}(t))$. Thus the sections $\{s_{i,t}\}_{i=0}^{N_\ell}$ converge to a set of holomorphic sections $\{s_{i,\infty}\}_{i=0}^{N_\ell}$ which forms another basis of $H^0(M;K_M^\ell)$. It induces a morphism
$$\Phi:M\rightarrow M_{can}\subset\mathbb{C}P^{N_\ell}.$$
On the other hand, we also have a uniform gradient estimate to each $s=s_{i,0}$, cf. \cite{So14-1},
$$|\nabla^{h(t)} s|_{h(t)\otimes\omega(t)}=|\nabla^{\Omega^{-\ell}} s+\partial u\otimes s|_{h(t)\otimes\omega(t)}
\le C|\nabla^{\Omega^{-\ell}} s|_{\Omega^{-\ell}\otimes\chi}+C|\partial u|_{\omega(t)}|s|_{\Omega^{-\ell}}\le C$$
where we used the uniform $C^1$ estimate of $u$ and $\chi\le C\omega(t)$, this leads to a convergence of $\{s_{i,t}\}_{i=0}^{N_\ell}$ under the Cheeger-Gromov convergence, so $\{s_{i,\infty}\}_{i=0}^{N_\ell}$ can also be seen as a set of holomorphic sections of $K_{M_\infty}^\ell$. It is obvious that $\{s_{i,\infty}\}_{i=0}^{N_\ell}$ is base point free on $M_\infty$, so it defines a map
$$\Phi_\infty:M_\infty\rightarrow\mathbb{C}P^{N_\ell}.$$
Finally, through construction of local peak sections, up to rising a power of $\ell$, the map $\Phi_\infty$ separates points of $M_\infty$, so it defines a homeomorphism.

\bibliographystyle{amsalpha}

\begin{thebibliography}{A}
\bibitem{An89} M. Anderson, {\it Ricci curvature bounds and Einstein metrics on compact manifolds}, J. Amer. Math. Soc. 2 (1989), 455-490.

\bibitem{An90} M. Anderson, {\it Convergence and rigidity of manifolds under Ricci curvature bounds}, Invent. Math., 102 (1990), 429-445.

\bibitem{BaZh15} R. H. Bamler and Q. S. Zhang, {\it Heat kernel and curvature bounds in Ricci flows with bounded scalar curvature}, arXiv:1501.02191v2

\bibitem{BKN} S. Bando, A. Kasue and H. Nakajima, {\it On a construction of coordinates at infinity on manifolds with fast curvature decay and maximal volume growth}, Invent. Math. 97 (1989), no. 2, 313¨C349.

\bibitem{Besse} A. L. Besse, Einstein manifolds, Springer-Verlag Berlin Herdelberg GmbH

\bibitem{Cao85} H.D. Cao, {\it Deformation of K\"ahler metrics to K\"ahler-Einstein metrics on compact K\"ahler manifolds}, Invent. Math., 81 (1985), 359-372.

%\bibitem{Cao11} X.D. Cao, {\it Curvature pinching estimate and singularities of the Ricci flow}, Comm. Anal. Geom., 19 (2011), 975-990.

\bibitem{Ch03} J. Cheeger, {\it Integral bounds on curvature, elliptic estimates and rectifiability of singular sets}, Geom. Funct. Anal., 13 (2003), 20-72.

\bibitem{ChCo96} J. Cheeger and T. H. Colding, {\it Lower bounds on the Ricci curvature and the almost rigidity of warped products}, Ann. Math., 144 (1996), 189-237.

\bibitem{ChCo97} J. Cheeger and T. H. Colding, {\it On the structure of spaces with Ricci curvature bounded below I}, J. Diff. Geom., 46 (1997), 406-480.

\bibitem{ChCo00} J. Cheeger and T. H. Colding, {\it On the structure of spaces with Ricci curvature bounded below II}, J. Diff. Geom., 54 (2000), 13-35.

\bibitem{ChCoTi} J. Cheeger, T. H. Colding and G. Tian, {\it On the singularities of spaces with bounded Ricci curvature}, Geom. Funct. Anal., 12 (2002), 873-914.

\bibitem{Co97} T. H. Colding, {\it Ricci curvature and volume convergece}, Ann. of Math., 145 (1997), 477-501.

\bibitem{CoNa12} T. H. Colding and A. Naber, {\it Sharp H\"older continuity of tangent cones for spaces with a lower Ricci curvature bound and applications}, Ann. of Math., 176 (2012), 1173-1229.

\bibitem{CoTo13} T. C. Collins and V. Tosatti, {\it K\"ahler currents and null loci}, arXiv:1304.5216v5

\bibitem{Cr80} C. B. Croke, {\it Some isoperimetric inequalities and eigenvalue estimates}, Ann. Sci. \'Ec. Norm. Sup. Paris, 13 (1980), 419-435.

\bibitem{DaWe} X.Z. Dai and G.F. Wei, {\it Comparison geometry for Ricci curvature}, preprint.

%\bibitem{DoSu14} S. Donaldson and S. Sun, {\it Gromov-Hausdorff limits of K\"ahler manifolds and algebraic geometry}, Acta Math. 213 (2014), no. 1, 63¨C106.

%\bibitem{KlLo08} B. Kleiler and J. Lott, {\it Notes on Perelman's papers}, Geom. Top., 12 (2008), 2587-2858.

\bibitem{Gr} M. Gromov, Metric structures for Riemannian and non-Riemannian spaces. With appendices by M. Katz, P. Pansu and S. Semmes. Translated from the French by Sean Michael Bates. Birkh\"auser Boston, Inc., Boston, MA, 2007.

\bibitem{Gu15} B. Guo, {\it On the K\"ahler Ricci flow on projective manifolds of general type}, arXiv:1501.04239

\bibitem{guosongwein} B. Guo, J. song and B. Weinkovw, {\it Geometric convergence of the K\"ahler-Ricci flow on complex surfaces of general type},  arXiv:1505.00705
%\bibitem{Mab86} T. Mabuchi, {\it K-energy maps integrating Fataki invariants}, T\^{o}hoko Math J., 38 (1986), 575-593.

\bibitem{Pe02} G. Perelman, {\it The entropy formula for the Ricci flow and its geometric applications}, arXiv:math.DG/0211159

\bibitem{PeWe97} P. Petersen and G.F. Wei, {\it Relative volume comparison with integral curvature bounds}, Geom. Funct. Anal., 7 (1997), 1031-1045.

\bibitem{PeWe01} P. Petersen and G.F. Wei, {\it Analysis and geometry on manifolds with integral Ricci curvature bounds, II}, Trans. AMS., 353 (2001), 457-478.

\bibitem{RoZh11} X.C. Rong and Y.G. Zhang, {\it Continuity of extremal transitions and flops for Calabi-Yau manifolds}, J. Diff. Geom., 89 (2011), 233-269.

\bibitem{ChSh} S.C. Chen and M.C. Shaw, {\it Partial differential equations in several complex variables},  Amer. Math. Soc., Providence, RI; Intern. Press, Boston, MA, 2001.

\bibitem{Si15} M. Simon, {\it Some integral curvature estimates for the Ricci flow in four dimensions}, arXiv:1514.02623v1

\bibitem{So14-1} J. Song, {\it Riemannian geometry of K\"ahler-Einstein currents}, arXiv:1404.0445

\bibitem{So14-2} J. Song, {\it Riemannian geometry of K\"ahler-Einstein currents II: an analytic proof of Kawamata's base point free theorem}, arXiv:1409.8374


\bibitem{SoTi11} J. Song and G. Tian, {\it Bounding scalar curvature for global solutions of the K\"ahler-Ricci flow}, arXiv:1111.5681v1


\bibitem{Ti90} G. Tian, {\it On Calabi's conjecture for complex surfaces with positive first Chern class}. Invent. Math., 101, (1990), 101-172.

\bibitem{TiZha13} G. Tian and Z.L. Zhang, {\it Regularity of K\"ahler-Ricci flows on Fano manifolds}, arXiv:1310.5897

\bibitem{TiZh06} G. Tian and Z. Zhang, {\it On the K\"ahler-Ricci flow on projective manifolds of general type}, Chinese Ann. Math. Ser. B, 27 (2006), 179-192.

\bibitem{Ts88} H. Tsuji, {\it Existence and degeneration of K\"ahler-Einstein metrics on minimal algebraic varieties of general type}, Math. Ann., 281 (1988), 123-133.

%\bibitem{De92} D. Yang, {\it Convergence of Riemannian manifolds with integral bounds on curvature I}, Ann. Sci. \'Ec. Norm. Sup., 25 (1992), 77-105.

\bibitem{Zh07} Q.S. Zhang, {\it A uniform Sobolev inequality under Ricci flow}, Inter. math. Res. Notices, 2007.

\bibitem{ZhaY09} Y.G. Zhang, {\it Miyaoka-Yau inequality for minimal projective manifolds of general type}, Proc. Amer. Math. Soc. 137 (2009), no. 8, 2749¨C2754.

\bibitem{Zh09} Z. Zhang, {\it Scalar curvature bound for K\"ahler-Ricci flows over minimal manifolds of general type}, Int. Math. Res. Not., 2009, 3901-3912.

\bibitem{Zha07} Z.L. Zhang, {\it Compact blow-up limits of finite time singularities of Ricci flow are shrinking Ricci solitons}, C. R. Acad. Sci. Paris, Ser. I 345 (2007), 503-506.
\end{thebibliography}

\end{document}